\documentclass[11pt]{amsart}

\newtheorem{theorem}{Theorem}[section]
\newtheorem{lem}[theorem]{Lemma}
\newtheorem{prop}[theorem]{Proposition}
\newtheorem{cor}[theorem]{Corollary}

\theoremstyle{definition}
\newtheorem{definition}[theorem]{Definition}

\theoremstyle{remark}

\numberwithin{equation}{section}

\begin{document}

\newcommand{\spacing}[1]{\renewcommand{\baselinestretch}{#1}\large\normalsize}
\spacing{1.14}

\title{Some Berwald Spaces of Non-positive flag Curvature}

\author {H. R. Salimi Moghaddam}

\address{Department of Mathematics, Faculty of  Sciences, University of Isfahan, Isfahan,81746-73441-Iran.} \email{salimi.moghaddam@gmail.com and hr.salimi@sci.ui.ac.ir}



\begin{abstract}
In this paper by using left invariant Riemannian metrics on some
3-dimensional Lie groups we construct some complete non-Riemannian
Berwald spaces of non-positive flag curvature and several families
of geodesically complete locally Minkowskian spaces of zero
constant flag curvature.
\end{abstract}

\maketitle


\section{\textbf{Introduction}}\label{intro}
The study of Riemannian manifolds with some curvature properties
is one of interesting problems in differential geometry. The
family of Lie groups equipped with invariant Riemannian metrics is
one of these important spaces. Riemannian Lie groups with some
sectional curvature properties have been studied by many
mathematicians and many interesting and important results have been
found on these manifolds (for example see \cite{Mi} and \cite{No}.).\\
In the recent years with extension of studies on Finsler manifolds
and applications of Finsler spaces in physics, Finsler spaces with
some curvature properties have been considered by many Finsler
geometers (\cite{BaChSh}, \cite{BaRo}, \cite{BaSh}, \cite{BeFa}
and \cite{Sh2}). Invariant Finsler metrics on Lie groups and
homogeneous spaces are of the best spaces for finding spaces with
some curvature properties. Some of these metrics and their flag
curvatures have been studied in \cite{DeHo1}, \cite{DeHo2},
\cite{EsSa1}, \cite{EsSa2}, \cite{Sa1} and \cite{Sa2}. An
important family of Finsler manifolds which can help us with
studying Finsler geometry is the family of flat Finsler spaces
(that is Finsler manifolds of zero constant flag curvatur). In the
present paper, after some preliminaries, by using left invariant
Riemannian metrics and left invariant vector fields on some
3-dimensional Lie groups we construct some geodesically complete
left invariant locally Minkowskian Randers metrics with zero
constant flag curvature. Some Finsler spaces with non-positive
flag curvature have been studied by Z. Shen in \cite{Sh2}. He
showed that every Finsler metric with negative flag curvature and
constant $S$-curvature must be Riemannian if the manifold is
compact. Also S. Deng and Z. Hou have studied homogeneous Finsler
spaces of non-positive curvature. They proved that a homogeneous
Finsler space with non-positive flag curvature and strictly
negative Ricci scalar is a simply connected manifold \cite{DeHo3}.
In the last section of the present paper we give some complete
left invariant non-Riemannian Berwald spaces of non-positive flag
curvature.

\section{\textbf{Preliminaries}}
Let $M$ be a smooth $n-$dimensional manifold and $TM$ be its
tangent bundle. A Finsler metric on $M$ is a non-negative function
$F:TM\longrightarrow \Bbb{R}$ which has the following properties:
\begin{enumerate}
    \item $F$ is smooth on the slit tangent bundle
    $TM^0:=TM\setminus\{0\}$,
    \item $F(x,\lambda y)=\lambda F(x,y)$ for any $x\in M$, $y\in T_xM$ and $\lambda
    >0$,
    \item the $n\times n$ Hessian matrix $[g_{ij}(x,y)]=[\frac{1}{2}\frac{\partial^2 F^2}{\partial y^i\partial
    y^j}]$ is positive definite at every point $(x,y)\in TM^0$.
\end{enumerate}
An important family of Finsler metrics introduced by G. Randers
(\cite{Ra}) in 1941 is Randers metrics with following form:
\begin{eqnarray}
   F(x,y)=\sqrt{g_{ij}(x)y^iy^j}+b_i(x)y^i,
\end{eqnarray}
where $g=(g_{ij}(x))$ is a Riemannian metric and $b=(b_i(x))$ is a
nowhere zero 1-form on $M$. It has been shown that $F$ is a
Finsler metric if and only if $\|b\|=b_i(x)b^i(x)<1$, where
$b^i(x)=g^{ij}(x)b_j(x)$ and $[g^{ij}(x)]$ is the inverse matrix
of $[g_{ij}(x)]$.\\
For constructing Randers metrics we can also use vector fields as
follows:
\begin{eqnarray}\label{F}
  F(x,y)=\sqrt{g(x)(y,y)}+g(x)(X(x),y),
\end{eqnarray}
where $X$ is a vector field on $M$ such that $\|X\|=\sqrt{g(X,X)}<1$.\\
A Riemannian Metric $g$ on the Lie group $G$ is called left
invariant if
\begin{eqnarray}
  g(x)(y,z)=g(e)(T_xl_{x^{-1}}y,T_xl_{x^{-1}}z) \ \ \ \ \forall x\in
  G, \forall y,z\in T_xG,
\end{eqnarray}
where $e$ is the unit element of $G$.\\
Similar to the Riemannian case, a Finsler metric is called left
invariant if
\begin{eqnarray}
  F(x,y)=F(e,T_xl_{x^{-1}}y).
\end{eqnarray}

For constructing left invariant Randers metrics on Lie groups we
can use left invariant Riemannian metrics and invariant vector
fields. Suppose that $G$ is a Lie group, $g$ is a left invariant
Riemannian metric and $X$ is a left invariant vector field such
that $\sqrt{g(X,X)}<1$, then we can define $F$ as the formula
\ref{F}. Clearly $F$ is left invariant. If $X$ be parallel with
respect to the Levi-Civita connection induced by the Riemannian
metric $g$
then $F$ is of Berwald type.\\
Flag curvature, which is a generalization of the concept of
sectional curvature in Riemannian geometry, is one of the
fundamental quantities which associates with a Finsler space. Flag
curvature is computed by the following formula:
\begin{eqnarray}
  K(P,Y)=\frac{g_Y(R(U,Y)Y,U)}{g_Y(Y,Y).g_Y(U,U)-g_Y^2(Y,U)},
\end{eqnarray}
where $g_Y(U,V)=\frac{1}{2}\frac{\partial^2}{\partial s\partial
t}(F^2(Y+sU+tV))|_{s=t=0}$, $P=span\{U,Y\}$,
$R(U,Y)Y=\nabla_U\nabla_YY-\nabla_Y\nabla_UY-\nabla_{[U,Y]}Y$ and
$\nabla$ is the Chern connection induced by $F$ (see \cite{BaChSh}
and \cite{Sh1}.).


\section{\textbf{Some geodesically complete locally Minkowskian  spaces of zero constant flag curvature}}

We begin this section with a study about completeness of a special
family of Randers metrics.
\begin{definition}
The Riemannian manifold $(M,g)$ is said to be homogeneous if the
group of isometries of $M$ acts transitively on $M$ (see
\cite{BeEhEa}.).
\end{definition}
\begin{theorem}
Let $(M,g)$ be a homogeneous Riemannian manifold. Suppose that $F$
is a Randers metric of Berwald type defined by $g$ and a 1-form
$b$. Then $(M,F)$ is geodesically complete.
\end{theorem}

\begin{proof}
Since $F$ is of Berwald type therefore the Chern connection of $F$
and the Levi-Civita connection of $g$ coincide and hence their
geodesics coincide. On the other hand $(M,g)$ is a homogeneous
Riemannian manifold, hence $(M,g)$ is geodesically complete (see
\cite{BeEhEa} page 185.). Therefore $(M,F)$ is geodesically
complete.
\end{proof}

\begin{cor}
In above theorem, suppose that $M$ is connected then by using
Hofp-Rinow theorem for Finsler manifolds, $(M,F)$ is complete.
\end{cor}

\begin{cor}
Let $G$ be a Lie group and $g$ be a left invariant Riemannian
metric on $G$. Also suppose that $X$ is a parallel vector field
with respect to the Levi-Civita connection of $g$ such that
$\|X\|<1$. Then the Randers metric defined by $g$, $X$ and the
relation \ref{F} is geodesically complete.
\end{cor}

Finding Riemannian spaces of constant sectional curvature is an
interesting problem in Riemannian geometry. A class of Riemannian
spaces of constant sectional curvature is the class of flat
manifolds (that is Riemannian spaces of zero constant sectional
curvature). Lie groups equipped with invariant Riemannian metrics
are suitable manifolds for finding spaces of constant sectional
curvature. For example Abelian Lie groups with left invariant
Riemannian metrics are flat. Therefore we can have the following
theorem:

\begin{theorem}
Let $G$ be an abelian Lie group equipped with a left invariant
Riemannian metric $g$ and let $\frak{g}$ be the Lie algebra of
$G$. Suppose that $X\in\frak{g}$ is a left invariant vector field
with $\sqrt{g(X,X)}<1$. Then the Randers metric $F$ defined by the
formula \ref{F} is a flat geodesically complete locally
Minkowskian metric on $G$.
\end{theorem}

\begin{proof}
Assume that $U,V,W\in\frak{g}$, now by using the formula
\begin{eqnarray}\label{nabla}
  2g(\nabla_UV,W)=g([U,V],W)-g([V,W],U)+g([W,U],V),
\end{eqnarray}
and the fact that $G$ is abelian we have $\nabla_YX=0$ for any
$Y\in\frak{g}$. Hence $X$ is parallel with respect to $\nabla$ and
$F$ is of Berwald type. Also the curvature tensor $R=0$ of $g$
coincides on the curvature tensor of $F$ and therefore the flag
curvature of $F$ is zero. $F$ is a flat Berwald metric therefore
by proposition 10.5.1 (page 275) of \cite{BaChSh}, $F$ is locally
Minkowskian.
\end{proof}

A famous class of Lie groups is the class of unimodular Lie
groups. A Lie group $G$ is said to be unimodular if its left
invariant Haar measure is also right invariant (see \cite{InoVek}
and \cite{Mi}.).

\begin{prop}\label{ccc}
Let $G$ be a 3-dimensional unimodular Lie group with a left
invariant Riemannian metric. Then there exists an orthonormal
basis $\{x,y,z\}$ of the Lie algebra $\frak{g}$ such that
\begin{eqnarray}
  [x,y]=c_3z \ \ \ , \ \ \ [y,z]=c_1x \ \ \ , \ \ \ [z,x]=c_2y , \
  \ \ \ c_i=\Bbb{R}.
\end{eqnarray}
(see \cite{InoVek}.)
\end{prop}

\begin{lem}\label{Uparallel}
Let $G$ be a 3-dimensional unimodular Lie group with a left
invariant Riemannian metric $g$ and use the notations introduced
above. $G$ admits a parallel (with respect to the Levi-Civita
connection of $g$) left invariant vector field $U=u_1x+u_2y+u_3z$
if and only if the following equations hold:
\begin{eqnarray}\label{muparallel}
  \mu_1u_1=\mu_1u_2=\mu_2u_1=\mu_2u_3=\mu_3u_1=\mu_3u_2=0,
\end{eqnarray}
where $\mu_i=\frac{1}{2}(c_1+c_2+c_3)-c_i$, $i=1,2,3$.
\end{lem}

\begin{proof}
Suppose that $\nabla$ is the Levi-Civita connection of $g$. Let $<
, >$ be the inner product induced by $g$ on $\frak{g}$.  Then by
using formula \ref{nabla} and some computations we have (also you
can see \cite{InoVek}.):
\begin{eqnarray}
  \nabla_xx&=&0 \ \ , \ \ \nabla_xy=\mu_1z, \ \ \nabla_xz=-\mu_1y, \nonumber\\
  \nabla_yx&=&-\mu_2z \ \ , \ \ \nabla_yy=0, \ \ \nabla_yz=\mu_2x, \\
  \nabla_zx&=&\mu_3y \ \ , \ \ \nabla_zy=-\mu_3x \ \ , \ \ \nabla_zz=0\nonumber.
\end{eqnarray}
Now it is suffix to let $\nabla_x U=\nabla_y U=\nabla_z U=0$.
\end{proof}

With K. Nomizu \cite{No} we consider the Lie groups $\Bbb{H},
E(1,1)$ and $E(2)$ in the following theorem.

\begin{theorem}
Consider the following 3-dimensional Lie groups:
\begin{itemize}
    \item $\Bbb{H}$: the Heisenberg group,
    \item E(1,1): group of rigid motions of Minkowski 2-space,
    \item E(2): group of rigid motions of Euclidean 2-space.
\end{itemize}
$\Bbb{H}$ and $E(1,1)$ do not admit any non-Riemannian Berwald
Randers metric of the form \ref{F}, where $g$ and $X$ are left
invariant. $E(2)$ admits a non-empty family of flat geodesically
complete locally Minkowskian Randers metrics of the form \ref{F},
where $g$ and $X$ are left invariant.
\end{theorem}

\begin{proof}
\textbf{Case 1 ($\Bbb{H}$).} Let $\frak{h}$ be the Lie algebra of
$\Bbb{H}$ and $g$ be any left invariant Riemannian metric on
$\Bbb{H}$. By using proposition \ref{ccc} and \cite{InoVek} there
exist an orthonormal basis $\{x,y,z\}$for $\frak{h}$ such that
$c_1>0, c_2=c_3=0$. Therefore we have $\mu_1=-\frac{1}{2}c_1<0,
\mu_2=\mu_3=\frac{1}{2}c_1>0$. Now let $U=u_1x+u_2y+u_3z$ be any
left invariant vector field on $\Bbb{H}$ such that is parallel
with respect to the Levi-Civita connection $\nabla$ of $g$. By
using lemma \ref{Uparallel} we have $c_1u_1=c_1u_2=c_1u_3=0$ and
so $U=0$.\\

\textbf{Case 2 ($E(1,1)$).} Suppose that $\frak{e}(1,1)$ is the
Lie algebra of $E(1,1)$ and $g$ is a left invariant Riemannian
metric on $E(1,1)$. By using proposition \ref{ccc} and
\cite{InoVek} there exist an orthonormal basis $\{x,y,z\}$for
$\frak{e}(1,1)$ such that $c_1>0$, $c_2<0$ and $c_3=0$. Now we use
lemma \ref{Uparallel} and
get $(c_2-c_1)u_3=(c_2-c_1)u_2=(c_1-c_2)u_1=0$. Therefore $U=0$.\\

\textbf{Case 3 ($E(2)$).} We consider the Lie group $E(2)$ as
follows:
\begin{eqnarray}
  E(2)=\{\left[%
\begin{array}{ccc}
  \cos\theta & -\sin\theta & a \\
  \sin\theta & \cos\theta & b \\
  0 & 0 & 1 \\
\end{array}%
\right]| a,b,\theta\in\Bbb{R}\}.
\end{eqnarray}

The Lie algebra of $E(2)$ is of the form
\begin{eqnarray}
  \frak{e}(2)= span\{x=\left[%
\begin{array}{ccc}
  0 & 0 & 1 \\
  0 & 0 & 0 \\
  0 & 0 & 0 \\
\end{array}%
\right],y=\left[%
\begin{array}{ccc}
  0 & 0 & 0 \\
  0 & 0 & 1 \\
  0 & 0 & 0 \\
\end{array}%
\right],z=\left[%
\begin{array}{ccc}
  0 & -1 & 0 \\
  1 & 0 & 0 \\
  0 & 0 & 0 \\
\end{array}%
\right]\},
\end{eqnarray}
where
\begin{eqnarray}
  [x,y]=0 \ \ \ , \ \ \ [y,z]=x\ \ \ , \ \ \ [z,x]=y.
\end{eqnarray}
Now consider the following inner product and let $g$ be the left
invariant Riemannian metric induced by this inner product on
$E(2)$,
\begin{eqnarray}\label{inner product}
  <x,x>=<y,y>=<z,z>=\lambda^2 \ \ , \ \ <x,y>=<y,z>=<z,x>=0, \ \ \lambda>0.
\end{eqnarray}
By using formula \ref{nabla} and some computations for the
Levi-Civita connection $\nabla$ of $g$  we have:
\begin{eqnarray}
  \nabla_xx&=&0 \ \ , \ \ \nabla_xy=0, \ \ \nabla_xz=0, \nonumber\\
  \nabla_yx&=&0 \ \ , \ \ \nabla_yy=0, \ \ \nabla_yz=0, \\
  \nabla_zx&=&y \ \ , \ \ \nabla_zy=-x \ \ , \ \ \nabla_zz=0\nonumber.
\end{eqnarray}
For curvature tensor of $\nabla$ we have $R=0$ and so $(E(2),g)$
is a flat Riemannian manifold. Suppose that
$U=u_1x+u_2y+u_3z\in\frak{e}(2)$ such that $U$ is parallel with
respect to $\nabla$. A simple computation shows that $U=uz$.
Assume that $\sqrt{<U,U>}<1$, in other words let
$0<|u|<\frac{1}{\lambda}$. Hence, the left invariant Randers
metric $F$ defined by $g$ and $U$ with formula \ref{F} is of
Berwald type. Also since $F$ is of Berwald type therefore the
curvature tensor of $F$ and $g$ coincide and $F$ is of zero
constant flag curvature. Hence $F$ is locally Minkowskian.
\end{proof}
Another example of flat geodesically complete locally Minkowskian
Randers spaces is described as follows.\\

Let $\frak{g}=span\{x,y,z\}$ be a Lie algebra such that
\begin{eqnarray}
  [x,y]=\alpha y+\alpha z \ \ \ , \ \ \ [y,z]=2\alpha x \ \ \ , \ \ \ [z,x]=\alpha
  y+\alpha z \ \ \ , \ \ \ \alpha\in\Bbb{R}.
\end{eqnarray}
Also consider the inner product described by \ref{inner product}
on $\frak{g}$.\\
Suppose that $G$ is a Lie group with Lie algebra $\frak{g}$, and
$g$ is the left invariant Riemannian metric induced by the above
inner product $<.,.>$ on $G$.\\
A direct computation for the Levi-Civita connection $\nabla$ of
$(G,g)$ shows that:
\begin{eqnarray}
  \nabla_xx&=&0 \ \ , \ \ \nabla_xy=0, \ \ \nabla_xz=0, \nonumber\\
  \nabla_yx&=&-\alpha y-\alpha z \ \ , \ \ \nabla_yy=\alpha x, \ \ \nabla_yz=\alpha x, \\
  \nabla_zx&=&\alpha y+\alpha z \ \ , \ \ \nabla_zy=-\alpha x \ \ , \ \ \nabla_zz=-\alpha x\nonumber.
\end{eqnarray}
The above equations for $\nabla$ show that $R=0$, therefore
$(G,g)$ is a flat Riemannian manifold.\\
Let $U=u_1x+u_2y+u_3z\in\frak{g}$ be a left invariant vector field
such that $\nabla U=0$. By a short computation we have $U=uy-uz$.
Now suppose that $\sqrt{2}|u|\lambda=\sqrt{<U,U>}<1$ or
equivalently let $0<|u|<\frac{1}{\sqrt{2}\lambda}$. Therefore the
invariant Randers metric $F$ defined by $g$ and $U$ is a flat
geodesically complete locally Minkowskian metric on $G$. Also if
we consider $G$ is connected, $(G,F)$ will be complete.

\section{\textbf{Examples of complete Berwald spaces of non-positive flag curvature}}

In this section we give three families of complete Berwald spaces
of non-positive flag curvature also we give their explicit
formulas for computing flag curvature.\\

\textbf{Example 1.} Suppose that $\frak{g}=span\{x,y,z\}$ is a Lie
algebra with the following structure:
\begin{eqnarray}
  [x,y]=\alpha y+\alpha z \ \ , \ \ [y,z]=0 \ \ , \ \ [z,x]=-\alpha y-\alpha
  z \ \ , \ \ \alpha\in\Bbb{R}.
\end{eqnarray}
Also consider the inner product $<,>$ described by \ref{inner
product} on $\frak{g}$ with respect to the basis $\{x,y,z\}$.
Suppose that $G$ is a connected Lie group with Lie algebra
$\frak{g}$ and $g_1$ is the left invariant Riemannian metric induced by $<,>$ on $G$.\\
By using formula \ref{nabla} for the Levi-Civita connection of the
left invariant Riemannian metric $g_1$ we have:
\begin{eqnarray}
  \nabla_xx&=&0 \ \ , \ \ \nabla_xy=0 \ \ , \ \ \nabla_xz=0, \nonumber\\
  \nabla_yx&=&-\alpha y-\alpha z \ \ , \ \ \nabla_yy=\alpha x \ \ , \ \ \nabla_yz=\alpha x, \\
  \nabla_zx&=&-\alpha y-\alpha z \ \ , \ \ \nabla_zy=\alpha x \ \ , \ \ \nabla_zz=\alpha x\nonumber.
\end{eqnarray}
Also for curvature tensor of $(G,g_1)$ we have:
\begin{eqnarray}
  R(x,y)x&=&R(x,z)x= 2\alpha^2(y+z),\nonumber\\
  R(x,y)y&=&R(x,y)z=R(x,z)y=R(x,z)z=-2\alpha^2x \\
  R(y,z)x&=&R(y,z)y=R(y,z)z=0.\nonumber
\end{eqnarray}
Now let $Y=ax+by+cz$ and $V=\tilde{a}x+\tilde{b}y+\tilde{c}z$ be
two vectors in $\frak{g}$. A simple computation shows that
\begin{eqnarray}
  R(V,Y)Y=2\alpha^2(a(y+z)-x(b+c))(\tilde{a}(b+c)-a(\tilde{b}+\tilde{c})).
\end{eqnarray}
Suppose that $\{Y,V\}$ is orthonormal with respect to $<,>$, then
a direct computation for sectional curvature of $(G,g_1)$ shows
that:
\begin{eqnarray}
  K(V,Y)=-2(\alpha\lambda(a(\tilde{b}+\tilde{c})-\tilde{a}(b+c)))^2\leq0.
\end{eqnarray}
Therefore the Riemannian manifold $(G,g_1)$ is of non-positive
sectional curvature.\\
Now let $U=u_1x+u_2y+u_3z\in\frak{g}$ be a left invariant vector
field and $\nabla U=0$. So we have $U=uy-uz$. Assume that
$\|U\|=\sqrt{<U,U>}<1$ or equivalently let
$0<|u|<\frac{1}{\sqrt{2}\lambda}$.\\
Now consider the Randers metric $F_1$ defined by $g_1$ and $U$. By
using formula of $g_Y$ (Also you can see \cite{EsSa1}.) we have
the following equations:
\begin{eqnarray}
  g_Y(R(V,Y)Y,V)&=& -2(\alpha\lambda(\tilde{a}(b+c)-a(\tilde{b}+\tilde{c})))^2(1+(b-c)u\lambda^2)\nonumber \\
  g_Y(Y,Y)&=&(1+u\lambda^2(b-c))^2 \\
  g_Y(V,V)&=&1+u\lambda^2(y\lambda^2(\tilde{b}-\tilde{c})^2+(b-c))\nonumber\\
  g_Y(Y,V)&=&u\lambda^2(\tilde{b}-\tilde{c})(1+u\lambda^2(b-c))\nonumber.
\end{eqnarray}
A simple direct computation shows that $F_1$ is of non-positive
flag curvature and its flag curvature for the flag $P=span\{Y,V\}$
is obtained by the following formula:
\begin{eqnarray}
  K(P,Y)=\frac{-2(\alpha\lambda(\tilde{a}(b+c)-a(\tilde{b}+\tilde{c})))^2}{(1+u\lambda^2(b-c))^2}\leq0.
\end{eqnarray}

\textbf{Example 2.} The Lie algebra $\frak{g}$ in example 1 admits
a basis $\{x,y,z\}$ such that:
\begin{eqnarray}
  [x,y]=0 \ \ , \ \ [y,z]=0 \ \ , \ \ [z,x]=\alpha x \ \ , \ \ \alpha\in\Bbb{R}.
\end{eqnarray}
Now let $<,>$ be the inner product introduced by \ref{inner
product} with respect to this new basis. Also suppose that $g_2$
is the left invariant Riemannian metric induced by $<,>$ on $G$.
Similar to example 1 we can obtain the Levi-Civita connection of
$(G,g_2)$ as follows:
\begin{eqnarray}
  \nabla_xx&=&\alpha z \ \ , \ \ \nabla_xy=0 \ \ , \ \ \nabla_xz=-\alpha x, \nonumber\\
  \nabla_yx&=&0 \ \ , \ \ \nabla_yy=0 \ \ , \ \ \nabla_yz=0, \\
  \nabla_zx&=&0 \ \ , \ \ \nabla_zy=0 \ \ , \ \ \nabla_zz=0\nonumber.
\end{eqnarray}
Also for curvature tensor we have:
\begin{eqnarray}
  R(x,y)x=R(x,y)y=R(x,y)z&=&R(x,z)y=R(y,z)x=R(y,z)y=R(y,z)z=0,\nonumber\\
  R(x,z)x&=&\alpha^2z \ \ , \ \ R(x,z)z=-\alpha^2x.
\end{eqnarray}
Therefore for an orthonormal basis
$\{Y=ax+by+cz,V=\tilde{a}x+\tilde{b}y+\tilde{c}z\}$ we have:
\begin{eqnarray}
  K(V,Y)=-(\alpha\lambda(\tilde{a}c-\tilde{c}a))^2\leq0,
\end{eqnarray}
which shows $(G,g_2)$ is of non-positive sectional curvature.\\
It is easy to show that the only left invariant vector fields
parallel with respect to $\nabla$ are of the form $U=uy$,
$u\in\Bbb{R}$. Assume that
$0<|u|<\frac{1}{\lambda}$, so we have $\|U\|<1$.\\
Let $F_2$ be the Randers metric defined by $g_2$ and $U$. Then for
$F_2$ we have:
\begin{eqnarray}
  g_Y(R(V,Y)Y,V)&=& -(\alpha\lambda(\tilde{a}c-\tilde{c}a))^2(1+bu\lambda^2)\nonumber \\
  g_Y(Y,Y)&=&(1+ub\lambda^2)^2 \\
  g_Y(V,V)&=&1+(u\tilde{b}\lambda^2)^2+ub\lambda^2\nonumber\\
  g_Y(Y,V)&=&u\tilde{b}\lambda^2(1+ub\lambda^2)\nonumber.
\end{eqnarray}
Therefore $(G,F_2)$ is of non-positive flag curvature as follows:
\begin{eqnarray}
  K(P,Y)=\frac{-(\alpha\lambda(\tilde{a}c-\tilde{c}a))^2}{(1+ub\lambda^2)^2}\leq0,
\end{eqnarray}
where $P=span\{Y,V\}$.\\

\textbf{Example 3.} The Lie algebra $\frak{g}$ described in
example 1 also has a basis of the form $\{x,y,z\}$ such that:
\begin{eqnarray}
  [x,y]=0 \ \ , \ \ [y,z]=\alpha y+\alpha z \ \ , \ \ [z,x]=0 \ \ , \ \ \alpha\in\Bbb{R}.
\end{eqnarray}
Now consider the inner product defined by \ref{inner product} with
respect to this basis. Also let $g_3$ be the left invariant
Riemannian metric induced by this inner product. Similar to the
above examples for $(G,g_3)$ we have:
\begin{eqnarray}
  \nabla_xx&=&0 \ \ , \ \ \nabla_xy=0 \ \ , \ \ \nabla_xz=0, \nonumber\\
  \nabla_yx&=&0 \ \ , \ \ \nabla_yy=-\alpha z \ \ , \ \ \nabla_yz=\alpha y, \\
  \nabla_zx&=&0 \ \ , \ \ \nabla_zy=-\alpha z \ \ , \ \ \nabla_zz=\alpha
  y\nonumber,
\end{eqnarray}
and,
\begin{eqnarray}
  R(x,y)x=R(x,y)y=R(x,y)z&=&R(x,z)y=R(x,z)x=R(y,z)x=R(x,z)z=0,\nonumber\\
  R(y,z)y&=&2\alpha^2z \ \ , \ \ R(y,z)z=-2\alpha^2y.
\end{eqnarray}
Hence for an orthonormal basis
$\{Y=ax+by+cz,V=\tilde{a}x+\tilde{b}y+\tilde{c}z\}$ we have:
\begin{eqnarray}
  K(V,Y)=-2(\alpha\lambda(\tilde{b}c-\tilde{c}b))^2\leq0.
\end{eqnarray}
Therefore $(G,g_3)$ is of non-positive sectional curvature.\\
The only left invariant vector fields parallel with respect to
$\nabla$ are of the form $U=ux$, $u\in\Bbb{R}$. Let
$0<|u|<\frac{1}{\lambda}$, so we have $\|U\|<1$.\\
Let $F_3$ be the Randers metric defined by $g_3$ and $U$. Then for
$F_3$ we have:
\begin{eqnarray}
  g_Y(R(V,Y)Y,V)&=& -2(\alpha\lambda(\tilde{b}c-\tilde{c}b))^2(1+au\lambda^2)\nonumber \\
  g_Y(Y,Y)&=&(1+ua\lambda^2)^2 \\
  g_Y(V,V)&=&1+au\lambda^2+\tilde{a}^2u^2\lambda^4\nonumber\\
  g_Y(Y,V)&=&\tilde{a}u\lambda^2(1+au\lambda^2)\nonumber.
\end{eqnarray}
The above equations show that $(G,F_3)$ is of non-positive flag
curvature with the following formula:
\begin{eqnarray}
  K(P,Y)=\frac{-2(\alpha\lambda(\tilde{b}c-\tilde{c}b))^2}{(1+au\lambda^2)^2}\leq0,
\end{eqnarray}
where $P=span\{Y,V\}$.\\

All Finsler metrics introduced in examples 1, 2 and 3 are
geodesically complete and the connected Lie group $G$ with these
metrics is complete.


\bibliographystyle{amsplain}

\begin{thebibliography}{9}


\bibitem{BaChSh} D. Bao, S. S. Chern and Z. Shen, \emph{An Introduction to Riemann-Finsler
Geometry}, (Berlin: Springer) (2000).
\bibitem{BaRo} D. Bao, C. Robles, \emph{On Randers Spaces of Constant Flag Curvature},
Reports On Mathematical Physics \textbf{51(1)} (2003) 9--42.
\bibitem{BaSh} D. Bao, Z. Shen, \emph{Finsler Metrics of Constant Positive Curvature on Lie Group $S^3$},
J. London Math. Soc. \textbf{66(2)} (2002) 453--467.
\bibitem{BeEhEa} J. K. Beem, P. E. Ehrlich and K. L. Easley, \emph{Global Lorentzian Geometry},
(Marcel Dekker, INC.) (1996).
\bibitem{BeFa} A. Bejancu, H. R. Farran, \emph{Randers Manifolds of Positive Constant Curvature},
International J. Math. and Mathematical Sciences. (2003)
1155--1165.
\bibitem{DeHo1} S. Deng and Z. Hou, \emph{Invariant Finsler Metrics on Homogeneous Manifolds},
J. Phys. A: Math. Gen. \textbf{37} (2004), 8245--8253.
\bibitem{DeHo2} S. Deng, Z. Hou, \emph{Invariant Randers Metrics on Homogeneous Riemannian
Manifolds}, J. Phys. A: Math. Gen. \textbf{37} (2004), 4353--4360.
\bibitem{DeHo3} S. Deng, Z. Hou, \emph{Homogeneous Finsler spaces of negative curvature},
Journal of Geometry and Physics \textbf{57(2)} (2007), 657--664.
\bibitem{EsSa1} E. Esrafilian and H. R. Salimi Moghaddam, \emph{Flag Curvature of Invariant Randers
Metrics on Homogeneous Manifolds}, J. Phys. A: Math. Gen.
\textbf{39} (2006) 3319--3324.
\bibitem{EsSa2} E. Esrafilian and H. R. Salimi Moghaddam, \emph{Induced Invariant
Finsler Metrics on Quotient Groups}, Balkan Journal of Geometry
and Its Applications, Vol. \textbf{11}, No. 1 (2006) 73--79.
\bibitem{InoVek} J. I. Inoguchi and J. Van der Veken, \emph{A Complete Classification of Parallel
Surfaces in Three-Dimensional Homogeneous Spaces}, Geom. Dedicata,
\textbf{131} (2008) 159--172.
\bibitem{Mi} J. Milnor, \emph{Curvatures of Left Invariant Metrics on Lie Groups},
Advances in Mathematics. \textbf{21} (1976), 293--329.
\bibitem{No} K. Nomizu, \emph{Left-Invariant Lorentz Metrics On Lie Groups},
Osaka J. Math. \textbf{16} (1979), 143--150.
\bibitem{Ra} G. Randers, \emph{On an Asymmetrical Metric in the Four-Space of General Relativity},
Phys. Rev. \textbf{59}(1941), 195--199.
\bibitem{Sa1} H. R. Salimi Moghaddam, \emph{On the flag curvature of invariant Randers
metrics}, Math. Phys. Anal. Geom.  \textbf{11} (2008) 1--9.
\bibitem{Sa2} H. R. Salimi Moghaddam, \emph{Flag curvature of invariant $(\alpha,\beta)$-metrics
of type $\frac{(\alpha+\beta)^2}{\alpha}$}, J. Phys. A: Math.
Theor.  \textbf{41} (2008).
\bibitem{Sh1} Z. Shen, \emph{Lectures on Finsler Geometry}, (World Scientific) (2001).
\bibitem{Sh2} Z. Shen, \emph{Finsler Manifolds with NonPositive Flag Curvature and Constant $S$-Curvature},
Mathematische Zeitschrift \textbf{249}, (2005) 625-640.

\end{thebibliography}

\end{document}